\documentclass[12pt,reqno]{amsart}

\usepackage[T1]{fontenc}
\usepackage[utf8]{inputenc}
\usepackage[english]{babel}
\usepackage{amsthm,amssymb,amsmath}

\usepackage[top=1in, bottom=1in, left=1in, right=1in]{geometry}
\usepackage[colorlinks=true,linkcolor=blue,citecolor=blue,urlcolor=blue]{hyperref}

\newtheorem{theorem}{Theorem}[section]
\newtheorem{proposition}[theorem]{Proposition}
\newtheorem{corollary}[theorem]{Corollary}
\newtheorem{lemma}[theorem]{Lemma}

\theoremstyle{definition}
\newtheorem{definition}[theorem]{Definition}
\newtheorem{example}{Example}

\DeclareMathOperator{\Aut}{Aut}
\DeclareMathOperator{\Br}{Br}
\DeclareMathOperator{\cl}{cl}
\DeclareMathOperator{\IB}{IB}
\newcommand{\F}{\mathbb F}
\newcommand{\id}{\mathrm{id}}

\title[Inverse Baer deformations]{Inverse Baer deformations and finite simple skew braces of nilpotent type}

\author[Marco Damele]{Marco Damele}
\address{Department of Mathematics and Computer Science, University of Cagliari, Via Ospedale 72, 09124 Cagliari, Italy}
\email{marco.damele@unica.it}
\date{}


\begin{document}

\begingroup
\setlength{\abovedisplayskip}{1pt plus .5pt minus .5pt}
\setlength{\belowdisplayskip}{1pt plus .5pt minus .5pt}
\setlength{\abovedisplayshortskip}{0pt plus .5pt minus .5pt}
\setlength{\belowdisplayshortskip}{0pt plus .5pt minus .5pt}
\setlength{\jot}{0pt}
\setlength{\parskip}{0pt}

\begin{abstract}
Finite simple skew braces with non-abelian additive group remain largely unexplored. In particular, to the best of our knowledge, no finite simple skew brace with non-abelian nilpotent additive group had previously been constructed. We introduce an inverse Baer deformation which produces skew braces of additive nilpotency class at most two from braces of abelian type. Starting from a finite left brace $C=(A,\oplus,\circ)$ of odd order and a suitable biadditive alternating map $\kappa:A\times A\to A$, we define
\[
x+y=x\oplus y\oplus\frac12\kappa(x,y).
\]
The resulting skew brace $\IB_\kappa(C)$ satisfies $\Br(\IB_\kappa(C))=C$ and $[x,y]_+=\kappa(x,y)$, and simplicity passes from $C$ to $\IB_\kappa(C)$. Applying this construction to simple left braces arising from asymmetric products, for every pair of odd primes $p,q$ with $p\mid(q-1)$ and every $m\ge1$ we construct a finite simple skew brace $X_{p,q}^{(m)}$ of order $p^{2m(q-1)+1}q$
such that
\[
(X_{p,q}^{(m)},+)\cong E_{p,q}^{(m)}\times C_q,
\]
where $E_{p,q}^{(m)}$ is an extraspecial $p$-group of order $p^{2m(q-1)+1}$ and exponent $p$. Thus every fixed admissible pair $(p,q)$ yields infinitely many pairwise non-isomorphic finite simple skew braces with non-abelian nilpotent additive group of class two.
\end{abstract}

\maketitle

\section{Introduction}

Finite simple skew braces with non-abelian additive group are still very poorly understood. As Vendramin recently emphasized, ``almost nothing about simple skew braces is known when the additive group is not abelian''~\cite{Vendramin24}. This is in marked contrast with the abelian-type case, where several constructions of finite simple braces are available. Bachiller constructed the first non-trivial finite simple left braces~\cite{Bachiller18}, and large families were subsequently obtained through asymmetric products by Bachiller, Ced\'o, Jespers and Okni\'nski~\cite{BCJO19}. Outside the abelian-type setting, explicit examples are much scarcer. The smallest simple skew braces which are not braces have order $12$; there are two of them, with additive group $A_4$ and multiplicative group $C_3\rtimes C_4$~\cite{KSV21}. More recently, Byott constructed the first infinite family of simple skew braces which are not braces and which do not arise from non-abelian simple groups~\cite{Byott26}. His examples have order $p^pq$, where $q\mid(p^p-1)/(p-1)$, and both their additive and multiplicative groups are solvable.

The contrast becomes particularly sharp when nilpotency is imposed. On the multiplicative side, Damele and Ercan proved that if $B$ is a finite simple skew brace with nilpotent multiplicative group, then $B\cong\operatorname{Triv}(C_p)$ for some prime $p$~\cite{DE26}. Thus a non-trivial finite simple skew brace cannot have nilpotent multiplicative group. On the additive side, finite skew braces of nilpotent type, namely skew braces whose additive group is nilpotent, were systematically studied by Ced\'o, Smoktunowicz and Vendramin~\cite{CSV19}. To the best of our knowledge, prior to the present work no finite simple skew brace with non-abelian nilpotent additive group had been constructed. We show that such skew braces not only exist, but occur in infinite families already for every fixed admissible pair of primes.

The key point is a general construction which reverses the Baer transform for skew braces. The Baer transform of Ercan, G\"ul, G\"ulo\u{g}lu and K{\i}zmaz~\cite{EGGK26} associates with a finite skew brace $X=(X,+,\circ)$ of odd order and additive nilpotency class at most two the left brace $\Br(X)=(X,\oplus,\circ)$, where
\[
x\oplus y=x+y+\frac12[y,x]_+.
\]
Moreover, every ideal of $X$ remains an ideal of $\Br(X)$. This suggests reversing the construction. Starting from a finite left brace $C=(A,\oplus,\circ)$ of odd order, we deform its abelian addition by means of an alternating biadditive map $\kappa:A\times A\to A$.

\noindent\textbf{Theorem A.}
\emph{Let $C=(A,\oplus,\circ)$ be a finite left brace of odd order, and let $\kappa:A\times A\to A$ be biadditive and alternating. Put $K=\kappa(A,A)$ and assume that $\kappa(A,K)=\kappa(K,A)=0$ and
\[
\lambda_a^C(\kappa(x,y))=\kappa(\lambda_a^C(x),\lambda_a^C(y))
\]
for all $a,x,y\in A$. Then
\[
x+y=x\oplus y\oplus\frac12\kappa(x,y)
\]
defines a skew brace $\IB_\kappa(C)=(A,+,\circ)$ such that
\[
\cl(A,+)\leq2,\qquad \Br(\IB_\kappa(C))=C,\qquad [x,y]_+=\kappa(x,y).
\]
If $\kappa\neq0$, then $(A,+)$ is non-abelian of nilpotency class exactly two. Moreover, if $C$ is simple, then $\IB_\kappa(C)$ is simple.}

Thus inverse Baer deformation reduces the construction problem to finding a simple left brace carrying a non-zero alternating map satisfying the centrality and equivariance conditions above. The simplicity statement follows from the ideal-preservation property of the Baer transform: every ideal of $\IB_\kappa(C)$ is an ideal of $\Br(\IB_\kappa(C))=C$.

We then show that these hypotheses arise naturally in the asymmetric-product construction of~\cite{BCJO19}. Let $p$ and $q$ be odd primes with $p\mid(q-1)$, and let $R$ be the $(q-1)$-dimensional orthogonal $\F_p$-module used in \cite[Section~5.3]{BCJO19}, endowed with its invariant non-degenerate symmetric bilinear form $\beta$. For every $m\ge1$ we consider $V_m=R^{2m}.$
The orthogonal direct sum of $2m$ copies of $\beta$ gives the symmetric form required by the asymmetric-product construction and produces a simple left brace $C_{p,q}^{(m)}$. Pairing the copies of $R$ two at a time, we define
\[
\omega_m(u,v)=\sum_{j=1}^{m}
\bigl(\beta(u_{2j-1},v_{2j})-\beta(u_{2j},v_{2j-1})\bigr).
\]
This is a non-degenerate alternating form invariant under the relevant automorphisms and therefore yields a non-zero Baer-admissible map
\[
\kappa_m((u,a,s),(v,b,t))=(0,0,\omega_m(u,v)).
\]

\noindent\textbf{Theorem B.}
\emph{Let $p$ and $q$ be odd primes satisfying $p\mid(q-1)$. For every $m\ge1$ there exists a finite simple skew brace $X_{p,q}^{(m)}$ of order $|X_{p,q}^{(m)}|=p^{2m(q-1)+1}q$
whose additive group is
\[
(X_{p,q}^{(m)},+)\cong E_{p,q}^{(m)}\times C_q,
\]
where $E_{p,q}^{(m)}$ is an extraspecial $p$-group of order $p^{2m(q-1)+1}$ and exponent $p$. In particular, $\cl(X_{p,q}^{(m)},+)=2.$}

Theorem~B gives the first examples known to us of finite simple skew braces of non-abelian nilpotent type. Moreover, the phenomenon is far from sporadic. For every fixed pair of odd primes $p,q$ satisfying $p\mid(q-1)$, varying $m$ gives skew braces
\[
X_{p,q}^{(1)},X_{p,q}^{(2)},X_{p,q}^{(3)},\ldots
\]
of pairwise distinct orders $p^{2m(q-1)+1}q$. Hence every fixed admissible pair $(p,q)$ produces infinitely many pairwise non-isomorphic finite simple skew braces with non-abelian nilpotent additive group of class two.

The paper is organized as follows. Section~2 recalls the Baer transform and its ideal-preservation property. Section~3 introduces inverse Baer deformations and proves Theorem~A. Section~4 recalls the required asymmetric-product construction and constructs the simple left braces $C_{p,q}^{(m)}$. Section~5 constructs the invariant alternating forms $\omega_m$, applies the inverse Baer deformation, and proves Theorem~B and the resulting infinite-family statement.

\endgroup

\section{Preliminaries and the Baer transform}

\subsection{Skew braces}
We follow the terminology and basic notation on skew braces from
\cite{GV17,KSV21}. A \emph{skew brace} is a triple $(B,+,\circ)$ such that $(B,+)$ and $(B,\circ)$ are groups and $a\circ(b+c)=a\circ b-a+a\circ c$
for all $a,b,c\in B$. The identity elements of the two groups coincide. The map $\lambda^{B}:(B,\circ)\longrightarrow \Aut(B,+)$ given by $\lambda^{B}(a)(b)=\lambda^{B}_a(b)=-a+a\circ b,$
is a group homomorphism.
The skew brace is \emph{trivial} if $\lambda^{B}_a=\id$ for every $a\in B$, equivalently if $a\circ b=a+b$ for all $a,b\in B$. A subgroup $I\leq(B,+)$ is an \emph{ideal} if $I\trianglelefteq(B,+),$ $I\trianglelefteq(B,\circ)$ and for all $a \in B$ $\lambda^{B}_{a}(I)=I$.
A non-zero skew brace is \emph{simple} if its only ideals are $0$ and $B$. 

\subsection{The Baer transform}

We shall use the Baer transform of \cite{EGGK26}. Let $X=(X,+,\circ)$ be a
finite skew brace of odd order such that $(X,+)$ has nilpotency class at most
two. We use the commutator convention $ [x,y]_+=x+y-x-y.$
Since $(X,+)$ has odd order, every element has a unique half; we write
$\frac12 z$ for the unique element whose double is $z$. Baer's addition is
\begin{equation}\label{eq:baer-addition}
 x\oplus y=x+y+\frac12[y,x]_+.
\end{equation}
Ercan, G\"ul, G\"ulo\u{g}lu and K{\i}zmaz prove that $(X,\oplus)$ is abelian
and that $ \Br(X)=(X,\oplus,\circ)$
is a left brace \cite[Theorem~1.3]{EGGK26}.
The structural property that will be essential here is the following.

\begin{theorem}[Ercan--G\"ul--G\"ulo\u{g}lu--K{\i}zmaz]\label{thm:baer-ideals}
Let $X=(X,+,\circ)$ be a finite skew brace of odd order such that $(X,+)$ has
nilpotency class at most two. Every ideal of $X$ is an ideal of $\Br(X)$.
\end{theorem}

\begin{proof}
This is \cite[Proposition~2.2]{EGGK26}.
\end{proof}

Theorem~\ref{thm:baer-ideals} is one-sided: an ideal of $\Br(X)$ need not in
general be an ideal of $X$. For our purposes this is exactly the useful
direction. If a skew brace $X$ is constructed in such a way that
$\Br(X)$ is a prescribed simple left brace, then $X$ is automatically simple.

\section{Inverse Baer deformations}

We now isolate a general mechanism producing such skew braces. Let $C=(A,\oplus,\circ)$ be a finite left brace of odd order. Since
$(A,\oplus)$ has odd order, the doubling map is an automorphism and the symbol
$\frac12x$ is unambiguous for every $x\in A$.

\begin{definition}\label{def:admissible}
A map $ \kappa:A\times A\longrightarrow A$
is called \emph{Baer-admissible} for $C$ if the following conditions hold:
\begin{enumerate}
\item[(B1)] The map $\kappa:A\times A\to A$ is biadditive with respect to $\oplus$, that is, $\kappa(x\oplus y,z)=\kappa(x,z)\oplus\kappa(y,z)$ and $\kappa(x,y\oplus z)=\kappa(x,y)\oplus\kappa(x,z)$
for all $x,y,z\in A$, and it is alternating, meaning that $\kappa(x,x)=0$
for every $x\in A$.
\item[(B2)] if $K=\kappa(A,A)$, then $ \kappa(A,K)=\kappa(K,A)=0;$
\item[(B3)] $\kappa$ is equivariant under the lambda action of $C$, that is, $ \lambda_a^C(\kappa(x,y))
 =\kappa(\lambda_a^C(x),\lambda_a^C(y))$
for all $a,x,y\in A$.
\end{enumerate}
\end{definition}

Because $\kappa$ is alternating and biadditive, it is skew-symmetric: $ \kappa(x,y)=-\kappa(y,x).$
Given a Baer-admissible map $\kappa$, define a new operation on $A$ by
\begin{equation}\label{eq:inverse-baer-addition}
 x+y=x\oplus y\oplus\frac12\kappa(x,y).
\end{equation}
We denote the resulting candidate skew brace by
$ \IB_\kappa(C)=(A,+,\circ).$

\begin{theorem}[Inverse Baer deformation]\label{thm:inverse-baer}
Let $C=(A,\oplus,\circ)$ be a finite left brace of odd order and let $\kappa$
be Baer-admissible. Then the operation \eqref{eq:inverse-baer-addition} makes
$(A,+)$ a group of nilpotency class at most two, the multiplication $\circ$
together with $+$ makes $ X=\IB_\kappa(C)=(A,+,\circ)$
a skew brace, and $ \Br(X)=C.$
Moreover,
\begin{equation}\label{eq:commutator-kappa}
 [x,y]_+=\kappa(x,y)
\end{equation}
for all $x,y\in A$. Hence, if $\kappa\neq0$, then $(A,+)$ is non-abelian and
has nilpotency class exactly two.
\end{theorem}

\begin{proof}
\begingroup
\setlength{\abovedisplayskip}{3pt plus 1pt minus 1pt}
\setlength{\belowdisplayskip}{3pt plus 1pt minus 1pt}
\setlength{\abovedisplayshortskip}{2pt plus 1pt minus 1pt}
\setlength{\belowdisplayshortskip}{2pt plus 1pt minus 1pt}
\setlength{\jot}{2pt}

We first prove that $+$ is a group operation. Let $x,y,z\in A$. By definition,
$x+y=x\oplus y\oplus\frac12\kappa(x,y)$, and hence
\begin{align*}
(x+y)+z
&=(x+y)\oplus z\oplus\frac12\kappa(x+y,z)\\
&=x\oplus y\oplus z\oplus\frac12\kappa(x,y)
  \oplus\frac12\kappa\!\left(
  x\oplus y\oplus\frac12\kappa(x,y),z\right).
\end{align*}
By biadditivity,
\[
\kappa\!\left(x\oplus y\oplus\frac12\kappa(x,y),z\right)
=\kappa(x,z)\oplus\kappa(y,z)
 \oplus\kappa\!\left(\frac12\kappa(x,y),z\right).
\]
Since $\kappa(x,y)\in K$, condition \textup{(B2)} gives
$\kappa(\frac12\kappa(x,y),z)=0$. Therefore
\[
(x+y)+z
=x\oplus y\oplus z
 \oplus\frac12\kappa(x,y)
 \oplus\frac12\kappa(x,z)
 \oplus\frac12\kappa(y,z).
\]
Similarly,
\begin{align*}
x+(y+z)
&=x\oplus y\oplus z\oplus\frac12\kappa(y,z)
  \oplus\frac12\kappa\!\left(
  x,y\oplus z\oplus\frac12\kappa(y,z)\right)\\
&=x\oplus y\oplus z
  \oplus\frac12\kappa(y,z)
  \oplus\frac12\kappa(x,y)
  \oplus\frac12\kappa(x,z),
\end{align*}
where $\kappa(x,\frac12\kappa(y,z))=0$ by \textup{(B2)}. Since
$(A,\oplus)$ is abelian, the two expressions coincide, and $+$ is associative.

The zero element $0$ of $(A,\oplus)$ is also the identity for $+$. Indeed,
biadditivity gives $\kappa(x,0)=\kappa(0,x)=0$, and hence
\[
x+0=x\oplus0\oplus\frac12\kappa(x,0)=x,
\qquad
0+x=0\oplus x\oplus\frac12\kappa(0,x)=x.
\]
Let $-x$ denote the inverse of $x$ in $(A,\oplus)$. Since
\[
\kappa(x,-x)=-\kappa(x,x)=0
\]
by biadditivity and alternation, we obtain
\[
x+(-x)=x\oplus(-x)\oplus\frac12\kappa(x,-x)=0.
\]
Similarly, $(-x)+x=0$. Thus the inverse of $x$ with respect to $+$ coincides
with its inverse with respect to $\oplus$, and $(A,+)$ is a group.

We next compute its commutator. Since the inverses with respect to $+$ and
$\oplus$ coincide,
\[
[x,y]_+=x+y-x-y.
\]
First,
\begin{align*}
(x+y)+(-x)
&=x\oplus y\oplus\frac12\kappa(x,y)\oplus(-x)\\
&\qquad\oplus\frac12
\kappa\!\left(x\oplus y\oplus\frac12\kappa(x,y),-x\right).
\end{align*}
By biadditivity and \textup{(B2)},
\[
\kappa\!\left(x\oplus y\oplus\frac12\kappa(x,y),-x\right)
=\kappa(x,-x)\oplus\kappa(y,-x)=\kappa(x,y),
\]
where we used $\kappa(x,-x)=0$ and skew-symmetry. Hence
\[
(x+y)+(-x)=y\oplus\kappa(x,y).
\]
Adding $-y$, we get
\begin{align*}
[x,y]_+
&=\bigl(y\oplus\kappa(x,y)\bigr)+(-y)\\
&=y\oplus\kappa(x,y)\oplus(-y)
  \oplus\frac12\kappa\bigl(y\oplus\kappa(x,y),-y\bigr).
\end{align*}
Now
\[
\kappa\bigl(y\oplus\kappa(x,y),-y\bigr)
=\kappa(y,-y)\oplus\kappa(\kappa(x,y),-y)=0
\]
by alternation and \textup{(B2)}. Therefore
\begin{equation}\label{eq:commutator-kappa}
[x,y]_+=\kappa(x,y).
\end{equation}
Thus $(A,+)'$ is contained in $K$. If $k\in K$ and $a\in A$, then
$[a,k]_+=\kappa(a,k)=0$ by \textup{(B2)}. Hence
$K\leq Z(A,+)$ and $\cl(A,+)\leq2$. Equation
\eqref{eq:commutator-kappa} also shows that the class is exactly two when
$\kappa\neq0$.

It remains to verify the skew brace compatibility. Recall that the skew brace
identity
\[
a\circ(x+y)=a\circ x-_+a+a\circ y
\]
is equivalent to requiring
$\lambda_a^X\in\Aut(A,+)$ for every $a\in A$, where
$\lambda_a^X(x)=-_+a+(a\circ x)$; see, for instance, \cite{GV17}.

Fix $a\in A$ and put $\mu_a=\lambda_a^C\in\Aut(A,\oplus)$. Since
$C=(A,\oplus,\circ)$ is a left brace,
$a\circ x=a\oplus\mu_a(x)$. Using the fact that the inverses with respect
to $+$ and $\oplus$ coincide, we obtain
\begin{align}
\lambda_a^X(x)
&=(-a)+(a\circ x)\notag\\
&=(-a)\oplus(a\circ x)
  \oplus\frac12\kappa(-a,a\circ x)\notag\\
&=\mu_a(x)
  \oplus\frac12\kappa(-a,a)
  \oplus\frac12\kappa(-a,\mu_a(x))\notag\\
&=\mu_a(x)\oplus
  \left(-\frac12\kappa(a,\mu_a(x))\right).
\label{eq:new-lambda}
\end{align}
For $a\in A$, define
\[
N_a:A\longrightarrow A,
\qquad
N_a(x)=\frac12\kappa(a,x).
\]
By biadditivity, $N_a$ is an endomorphism of $(A,\oplus)$. Moreover,
\[
N_a^2(x)
=\frac12\kappa\left(a,\frac12\kappa(a,x)\right)=0
\]
by \textup{(B2)}, since $\kappa(a,x)\in K$. Hence $N_a^2=0$, and
\[
(\id-N_a)(\id+N_a)=\id-N_a^2=\id.
\]
Thus $\id-N_a$ is invertible, with inverse $\id+N_a$, and
\eqref{eq:new-lambda} becomes
\begin{equation}\label{eq:new-lambda-operator}
\lambda_a^X=(\id-N_a)\mu_a.
\end{equation}
In particular, $\lambda_a^X$ is bijective.

Put $F=\lambda_a^X=(\id-N_a)\mu_a$. Since $\mu_a$ and $N_a$ are
endomorphisms of $(A,\oplus)$, the map $F$ is $\oplus$-additive. We claim that
\[
F(\kappa(x,y))=\kappa(F(x),F(y)).
\]
Indeed,
\begin{align*}
\kappa(F(x),F(y))
&=\kappa\bigl(\mu_a(x)-N_a(\mu_a(x)),
              \mu_a(y)-N_a(\mu_a(y))\bigr)\\
&=\kappa(\mu_a(x),\mu_a(y)),
\end{align*}
because $N_a(\mu_a(x)),N_a(\mu_a(y))\in K$, so all terms involving one of
these elements vanish by \textup{(B2)}. By \textup{(B3)},
\begin{equation}\label{eq:kappa-F-one}
\kappa(F(x),F(y))
=\kappa(\mu_a(x),\mu_a(y))
=\mu_a(\kappa(x,y)).
\end{equation}
On the other hand,
\begin{align*}
F(\kappa(x,y))
&=(\id-N_a)\mu_a(\kappa(x,y))\\
&=\mu_a(\kappa(x,y))
-\frac12\kappa\bigl(a,\mu_a(\kappa(x,y))\bigr).
\end{align*}
Again by \textup{(B3)},
$\mu_a(\kappa(x,y))=\kappa(\mu_a(x),\mu_a(y))\in K$; hence the last term
vanishes by \textup{(B2)}. Therefore
\begin{equation}\label{eq:kappa-F-two}
F(\kappa(x,y))=\mu_a(\kappa(x,y)).
\end{equation}
Combining \eqref{eq:kappa-F-one} and \eqref{eq:kappa-F-two}, we obtain
$F(\kappa(x,y))=\kappa(F(x),F(y))$. Consequently,
\begin{align*}
F(x+y)
&=F\left(x\oplus y\oplus\frac12\kappa(x,y)\right)\\
&=F(x)\oplus F(y)\oplus\frac12F(\kappa(x,y))\\
&=F(x)\oplus F(y)\oplus\frac12\kappa(F(x),F(y))
 =F(x)+F(y).
\end{align*}
Thus $F$ is a bijective homomorphism of $(A,+)$, and therefore
$\lambda_a^X\in\Aut(A,+)$ for every $a\in A$. Hence
$X=(A,+,\circ)$ is a skew brace.

Finally, we determine its Baer transform. By
\eqref{eq:commutator-kappa} and skew-symmetry,
\[
[y,x]_+=\kappa(y,x)=-\kappa(x,y).
\]
If $u\in A$ and $k\in K$, then $\kappa(u,k)=0$ by \textup{(B2)}, and hence
\[
u+k=u\oplus k.
\]
Thus $+$ and $\oplus$ coincide whenever one of the summands belongs to $K$;
in particular, they coincide on $K$, and the half of an element of $K$ is
the same with respect to both operations. Since
$[y,x]_+=-\kappa(x,y)\in K$, the Baer addition is
\begin{align*}
x\oplus_{\Br}y
&=x+y+\frac12[y,x]_+\\
&=\left(x\oplus y\oplus\frac12\kappa(x,y)\right)
  +\left(-\frac12\kappa(x,y)\right)\\
&=x\oplus y
  \oplus\frac12\kappa(x,y)
  \oplus\left(-\frac12\kappa(x,y)\right)
 =x\oplus y.
\end{align*}
Hence the Baer addition of $X$ is precisely the original operation $\oplus$.
Since the multiplicative operation $\circ$ has not been changed, we conclude
that $\Br(X)=C$.

\endgroup
\end{proof}

The ideal-preservation theorem from \cite{EGGK26} immediately transfers
simplicity.

\begin{corollary}[Simplicity transfer]\label{cor:simplicity-transfer}
Let $C$ be a finite simple left brace of odd order and let $\kappa$ be
Baer-admissible for $C$. Then $\IB_\kappa(C)$ is simple.
\end{corollary}

\begin{proof}
Put $X=\IB_\kappa(C)$, and let $I$ be an ideal of $X$. By
Theorem~\ref{thm:baer-ideals}, $I$ is an ideal of $\Br(X)$. By
Theorem~\ref{thm:inverse-baer}, $ \Br(X)=C.$
Since $C$ is simple, either $I=0$ or $I=X$.
\end{proof}

Thus the problem of constructing simple skew braces with non-abelian
class-two additive group can be approached in two independent steps: find a
simple left brace, and then find a non-zero Baer-admissible alternating map on
its additive group.

\section{A simple asymmetric-product seed}

\begingroup
\setlength{\abovedisplayskip}{1pt plus .5pt minus .5pt}
\setlength{\belowdisplayskip}{1pt plus .5pt minus .5pt}
\setlength{\abovedisplayshortskip}{0pt plus .5pt minus .5pt}
\setlength{\belowdisplayshortskip}{0pt plus .5pt minus .5pt}
\setlength{\jot}{0pt}
\setlength{\parskip}{0pt}

We now construct the simple left braces which will serve as the input for the inverse Baer deformation. We use a particular case of the asymmetric-product construction of Bachiller, Ced\'o, Jespers and Okni\'nski \cite[Section~5.2]{BCJO19}. We first record the specialization needed below.

\begin{proposition}[Special case of {\cite[Section~5.2 and Theorem~5.7]{BCJO19}}]
\label{prop:BCJO-special}
Let $p$ and $q$ be distinct primes, let $\gamma\in\F_q^\times$ have order $p$, and let $V$ be a finite-dimensional $\F_p$-vector space endowed with a non-degenerate symmetric bilinear form $B:V\times V\to\F_p$. Assume that $\sigma,\tau\in\Aut_{\F_p}(V)$ satisfy $|\sigma|=q$, $|\tau|=p$,
\[
B(\sigma u,\sigma v)=B(u,v),\qquad B(\tau u,\tau v)=B(u,v)
\]
for all $u,v\in V$, and $\tau\sigma=\sigma^\gamma\tau$. Let $T=V\rtimes\F_q$ be the semidirect product of trivial braces in which $a\in\F_q$ acts on $V$ as $\sigma^a$. Thus $(T,+)=V\times\F_q$. The construction of \cite[Section~5.2]{BCJO19} yields an asymmetric-product left brace $C=T\rtimes_\circ\F_p$. Hence, as a set,
\[
C=V\times\F_q\times\F_p,
\]
with
\begin{equation}\label{eq:general-seed-addition}
(u,a,s)\oplus(v,b,t)=\bigl(u+v,\ a+b,\ s+t+B(u,v)\bigr),
\end{equation}
and
\begin{equation}\label{eq:general-seed-lambda}
\lambda^C_{(u,a,s)}(v,b,t)
=\bigl(\sigma^a\tau^s(v),\ \gamma^s b,\ t-B(\sigma^a\tau^s(v),u)\bigr).
\end{equation}
Moreover, $C$ is simple if
\[
\operatorname{Im}(\sigma-\id)+\operatorname{Im}(\tau-\id)=V.
\]
In particular, $C$ is simple whenever $\sigma-\id$ is invertible.
\end{proposition}
\begin{proof}
This is the case $s=1$, $l_1=q$ of \cite[Section~5.2]{BCJO19}. In the notation used there, the cyclic factor $\mathbb Z/(q)$ acts on $V$ through $\sigma$, while $\mathbb Z/(p)$ acts through $\tau$ and multiplication by $\gamma$ on $\mathbb Z/(q)$. The relation $\tau\sigma=\sigma^\gamma\tau$ is the required compatibility relation, while the invariance of $B$ means that $\sigma$ and $\tau$ belong to the orthogonal group of $B$. Since $\gamma\neq1$ and $q$ is prime, $\gamma-1$ is invertible in $\F_q$. The formulas above are the corresponding specialization of the asymmetric-product formulas, and \cite[Theorem~5.7]{BCJO19} gives the simplicity condition
\[
\operatorname{Im}(\sigma-\id)+\operatorname{Im}(\tau-\id)=V.
\]
Hence invertibility of $\sigma-\id$ is sufficient.
\end{proof}

We now construct data satisfying Proposition~\ref{prop:BCJO-special}. Let $p$ and $q$ be odd primes with $p\mid(q-1)$, and choose $\gamma\in\F_q^\times$ of order $p$. Put
\[
R=\F_p[x]/(1+x+\cdots+x^{q-1}),
\]
let $\xi$ denote the image of $x$ in $R$, and note that $\dim_{\F_p}R=q-1$. Define
\[
c(r)=\xi r,\qquad f(r(\xi))=r(\xi^\gamma).
\]
By \cite[Section~5.3]{BCJO19}, $c$ has order $q$, $f$ has order $p$, and
\begin{equation}\label{eq:fc-relation}
fc=c^\gamma f.
\end{equation}
Moreover, there is a non-degenerate symmetric bilinear form $\beta:R\times R\to\F_p$ characterized on $1,\xi,\ldots,\xi^{q-2}$ by
\begin{equation}\label{eq:beta}
\beta(\xi^i,\xi^j)=1-\delta_{ij}.
\end{equation}
Both $c$ and $f$ preserve $\beta$, and $c-\id_R$ is invertible.

Fix $m\ge1$ and set
\[
V_m=R^{2m},\qquad \sigma_m=c^{\oplus 2m},\qquad \tau_m=f^{\oplus 2m}.
\]
For $u=(u_1,\ldots,u_{2m})$ and $v=(v_1,\ldots,v_{2m})$ define
\begin{equation}\label{eq:B-form}
B_m(u,v)=\sum_{i=1}^{2m}\beta(u_i,v_i).
\end{equation}
Then $B_m$ is non-degenerate and symmetric, $|\sigma_m|=q$, $|\tau_m|=p$, and
\[
\tau_m\sigma_m=(fc)^{\oplus2m}=(c^\gamma f)^{\oplus2m}=\sigma_m^\gamma\tau_m.
\]
Since $c$ and $f$ preserve $\beta$,
\[
B_m(\sigma_m u,\sigma_m v)=B_m(u,v),\qquad
B_m(\tau_m u,\tau_m v)=B_m(u,v).
\]
Moreover,
\[
\sigma_m-\id_{V_m}=(c-\id_R)^{\oplus2m},
\]
so $\sigma_m-\id_{V_m}$ is invertible. Hence Proposition~\ref{prop:BCJO-special} applies.

Thus, first forming $T_m=V_m\rtimes\F_q$ with $a\in\F_q$ acting as $\sigma_m^a$, and then taking the asymmetric product with $\F_p$, we obtain a left brace
\[
C_{p,q}^{(m)}=T_m\rtimes_\circ\F_p
\]
whose underlying set is
\[
A_m=V_m\times\F_q\times\F_p.
\]
Its addition is
\begin{equation}\label{eq:seed-addition}
(u,a,s)\oplus(v,b,t)=\bigl(u+v,\ a+b,\ s+t+B_m(u,v)\bigr),
\end{equation}
and its lambda maps are
\begin{equation}\label{eq:seed-lambda}
\lambda^{C_{p,q}^{(m)}}_{(u,a,s)}(v,b,t)
=\bigl(\sigma_m^a\tau_m^s(v),\ \gamma^s b,\
t-B_m(\sigma_m^a\tau_m^s(v),u)\bigr).
\end{equation}
Notice that $A_m=V_m\times\F_q\times\F_p$ is the underlying set; the additive law is the twisted law \eqref{eq:seed-addition}. Since $B_m$ is symmetric, $(A_m,\oplus)$ is abelian, as expected for a left brace.

\begin{proposition}\label{prop:seed-simple}
For every $m\ge1$, the left brace $C_{p,q}^{(m)}$ is simple and
\[
|C_{p,q}^{(m)}|=p^{2m(q-1)+1}q.
\]
\end{proposition}
\begin{proof}
Since $\sigma_m-\id_{V_m}$ is invertible, Proposition~\ref{prop:BCJO-special} gives simplicity. Since $\dim_{\F_p}V_m=2m(q-1)$,
\[
|C_{p,q}^{(m)}|=|V_m|qp=p^{2m(q-1)+1}q.
\]
\end{proof}

\endgroup

\section{The alternating deformation and the infinite family}

\begingroup
\setlength{\abovedisplayskip}{1pt plus .5pt minus .5pt}
\setlength{\belowdisplayskip}{1pt plus .5pt minus .5pt}
\setlength{\abovedisplayshortskip}{0pt plus .5pt minus .5pt}
\setlength{\belowdisplayshortskip}{0pt plus .5pt minus .5pt}
\setlength{\jot}{0pt}
\setlength{\parskip}{0pt}

We now deform the abelian addition of $C_{p,q}^{(m)}$. The even number of copies of $R$ allows us to associate to the symmetric form $\beta$ a non-degenerate alternating form on $V_m=R^{2m}$. Define
\begin{equation}\label{eq:omega}
\omega_m(u,v)
=\sum_{j=1}^{m}
\bigl(\beta(u_{2j-1},v_{2j})-\beta(u_{2j},v_{2j-1})\bigr).
\end{equation}

\begin{lemma}\label{lem:omega}
The form $\omega_m:V_m\times V_m\to\F_p$ is alternating and non-degenerate. Moreover,
\[
\omega_m(\sigma_m u,\sigma_m v)=\omega_m(u,v),\qquad
\omega_m(\tau_m u,\tau_m v)=\omega_m(u,v)
\]
for all $u,v\in V_m$.
\end{lemma}
\begin{proof}
Since $\beta$ is symmetric, each summand in \eqref{eq:omega} vanishes when $u=v$, so $\omega_m$ is alternating. Suppose $u=(u_1,\ldots,u_{2m})$ belongs to the radical of $\omega_m$. Fix $j$. Taking $v$ with only its $(2j)$-th coordinate nonzero gives
\[
\beta(u_{2j-1},v_{2j})=0
\]
for every $v_{2j}\in R$, hence $u_{2j-1}=0$. Taking $v$ with only its $(2j-1)$-th coordinate nonzero gives
\[
\beta(u_{2j},v_{2j-1})=0
\]
for every $v_{2j-1}\in R$, hence $u_{2j}=0$. Thus $u=0$, and $\omega_m$ is non-degenerate. Finally, since both $c$ and $f$ preserve $\beta$, each summand in \eqref{eq:omega} is preserved by $\sigma_m$ and $\tau_m$, proving the two invariance identities.
\end{proof}

Define
\begin{equation}\label{eq:kappa-family}
\kappa_m:A_m\times A_m\longrightarrow A_m,\qquad
\kappa_m((u,a,s),(v,b,t))=(0,0,\omega_m(u,v)).
\end{equation}

\begin{proposition}\label{prop:kappa-admissible}
The map $\kappa_m$ defined in \eqref{eq:kappa-family} is Baer-admissible for $C_{p,q}^{(m)}$.
\end{proposition}
\begin{proof}
Biadditivity with respect to $\oplus$ follows from the biadditivity of $\omega_m$, since the $V_m$-component of \eqref{eq:seed-addition} is the ordinary sum in $V_m$, while alternation follows from Lemma~\ref{lem:omega}. Moreover,
\[
\kappa_m(A_m,A_m)\subseteq Z_0:=\{(0,0,z):z\in\F_p\},
\]
and, since $\kappa_m$ depends only on the $V_m$-coordinates,
\[
\kappa_m(A_m,Z_0)=\kappa_m(Z_0,A_m)=0.
\]
Thus \textup{(B2)} holds. It remains to verify lambda-equivariance. Let $g=(u,a,s)$, $x=(v,b,t)$ and $y=(w,d,r)$. By \eqref{eq:seed-lambda}, the $V_m$-components of $\lambda_g(x)$ and $\lambda_g(y)$ are $Lv$ and $Lw$, where $L=\sigma_m^a\tau_m^s$. Since $\sigma_m$ and $\tau_m$ preserve $\omega_m$, so does $L$, and hence
\[
\kappa_m(\lambda_g(x),\lambda_g(y))
=(0,0,\omega_m(Lv,Lw))
=(0,0,\omega_m(v,w)).
\]
On the other hand, \eqref{eq:seed-lambda} gives $\lambda_g(0,0,z)=(0,0,z)$ for every $z\in\F_p$. Therefore
\[
\lambda_g(\kappa_m(x,y))
=(0,0,\omega_m(v,w))
=\kappa_m(\lambda_g(x),\lambda_g(y)),
\]
so \textup{(B3)} holds.
\end{proof}

We may now apply the inverse Baer deformation theorem.

\begin{theorem}\label{thm:family}
Let $p$ and $q$ be odd primes satisfying $p\mid(q-1)$. For every $m\ge1$,
\[
X_{p,q}^{(m)}:=\IB_{\kappa_m}(C_{p,q}^{(m)})
\]
is a finite simple skew brace of order
\[
|X_{p,q}^{(m)}|=p^{2m(q-1)+1}q.
\]
Moreover,
\[
(X_{p,q}^{(m)},+)\cong E_{p,q}^{(m)}\times C_q,
\]
where $E_{p,q}^{(m)}$ is an extraspecial $p$-group of order $p^{2m(q-1)+1}$ and exponent $p$. In particular, $(X_{p,q}^{(m)},+)$ is non-abelian and nilpotent of class two.
\end{theorem}
\begin{proof}
By Proposition~\ref{prop:kappa-admissible} and Theorem~\ref{thm:inverse-baer}, the operation
\begin{equation}\label{eq:family-addition}
x+y=x\oplus y\oplus\frac12\kappa_m(x,y)
\end{equation}
defines a skew brace $X_{p,q}^{(m)}$ with
\[
\Br(X_{p,q}^{(m)})=C_{p,q}^{(m)}.
\]
Since $C_{p,q}^{(m)}$ is simple by Proposition~\ref{prop:seed-simple}, Corollary~\ref{cor:simplicity-transfer} implies that $X_{p,q}^{(m)}$ is simple. The underlying set is unchanged, so
\[
|X_{p,q}^{(m)}|=p^{2m(q-1)+1}q.
\]

Using \eqref{eq:seed-addition} and \eqref{eq:kappa-family}, the new additive law is
\begin{equation}\label{eq:explicit-new-addition}
(u,a,s)+(v,b,t)
=\left(u+v,\ a+b,\ s+t+B_m(u,v)+\frac12\omega_m(u,v)\right).
\end{equation}
The $\F_q$-coordinate is independent of the other coordinates, hence
\[
(X_{p,q}^{(m)},+)\cong P_m\times C_q,
\]
where $P_m$ is the group on $V_m\times\F_p$ with multiplication
\[
(u,s)(v,t)
=\left(u+v,\ s+t+B_m(u,v)+\frac12\omega_m(u,v)\right).
\]
Set
\[
Q_m(u)=\frac12B_m(u,u).
\]
Since $B_m$ is symmetric,
\[
Q_m(u+v)-Q_m(u)-Q_m(v)=B_m(u,v).
\]
Thus the change of coordinates $(u,s)\mapsto(u,s-Q_m(u))$ identifies $P_m$ with the group on $V_m\times\F_p$ having multiplication
\begin{equation}\label{eq:heisenberg-law}
(u,s)(v,t)=\left(u+v,\ s+t+\frac12\omega_m(u,v)\right).
\end{equation}
From \eqref{eq:heisenberg-law},
\[
[(u,s),(v,t)]=(0,\omega_m(u,v)).
\]
Since $\omega_m$ is non-degenerate, it is nonzero, and its image is therefore all of the one-dimensional space $\F_p$. Hence
\[
P_m'=\{(0,z):z\in\F_p\}.
\]
Moreover, $(u,s)$ is central if and only if $\omega_m(u,v)=0$ for every $v\in V_m$, which by non-degeneracy is equivalent to $u=0$. Thus
\[
Z(P_m)=P_m'=\{(0,z):z\in\F_p\}\cong C_p.
\]
Since $\omega_m(u,u)=0$, repeated use of \eqref{eq:heisenberg-law} gives $(u,s)^k=(ku,ks)$ for every integer $k$, and hence $(u,s)^p=(0,0)$. Thus $P_m$ has exponent $p$ and
\[
\Phi(P_m)=P_m'P_m^p=P_m'=Z(P_m).
\]
Therefore $P_m$ is extraspecial. Finally,
\[
|P_m|=|V_m|p=p^{2m(q-1)+1}.
\]
Writing $P_m=E_{p,q}^{(m)}$, we obtain
\[
(X_{p,q}^{(m)},+)\cong E_{p,q}^{(m)}\times C_q,
\]
and the additive group is non-abelian and nilpotent of class two.
\end{proof}
\begin{corollary}\label{cor:infinite-family}
There exist infinitely many pairwise non-isomorphic finite simple skew braces
whose additive groups are non-abelian nilpotent groups of class two.
\end{corollary}
\begin{proof}
Fix odd primes $p$ and $q$ with $p\mid(q-1)$. For every $m\ge1$,
Theorem~\ref{thm:family} gives a finite simple skew brace
$X_{p,q}^{(m)}$ of order $p^{2m(q-1)+1}q.$
These orders are distinct as $m$ varies, so the skew braces
$X_{p,q}^{(m)}$ are pairwise non-isomorphic.
\end{proof}

\endgroup

\begin{example}
Take $p=3$ and $q=7$. Since $3\mid(7-1)$, Theorem~\ref{thm:family}
gives, for every $m\ge1$, a finite simple skew brace $X_{3,7}^{(m)}$
of order $|X_{3,7}^{(m)}|
=3^{12m+1}\cdot 7.$
Moreover, $(X_{3,7}^{(m)},+)
\cong E_{3,7}^{(m)}\times C_7,$
where $E_{3,7}^{(m)}$ is an extraspecial $3$-group of order $3^{12m+1}$
and exponent $3$. In particular, $(X_{3,7}^{(m)},+)$ is non-abelian
and nilpotent of class two.
For instance, when $m=1$ we obtain a finite simple skew brace of order $3^{13}\cdot7$
whose additive group is $E\times C_7,$
where $E$ is an extraspecial $3$-group of order $3^{13}$ and exponent $3$.
Thus the sequence
\[
X_{3,7}^{(1)},\ X_{3,7}^{(2)},\ X_{3,7}^{(3)},\ldots
\]
gives an explicit infinite family of pairwise non-isomorphic finite simple
skew braces of non-abelian nilpotent type.
\end{example}

\end{document}